\documentclass[12pt,letterpaper]{article}
\usepackage{cite}
\usepackage{amsmath,amssymb,amsthm}
\usepackage{algorithm}
\usepackage{algpseudocode}
\usepackage{enumerate}
\usepackage[shortlabels]{enumitem}
\usepackage{pgf,tikz}
\usepackage{thmtools}
\usepackage{lscape}
\usepackage{subcaption}
\usepackage{hyperref}
\usepackage{mathrsfs}
\usepackage{enumitem}
\usepackage{multicol}

\usepackage[all]{xy}
\usepackage{mathtools}
\usepackage[utf8]{inputenc}
\usepackage[T1]{fontenc}
\usepackage{lmodern}
\usepackage{ulem} 
\usepackage{exscale,relsize}
\usepackage{mathtools}
\usepackage{xcolor}
\usepackage{amscd}

\usetikzlibrary{arrows,matrix}
\usetikzlibrary{shapes.geometric}

\newtheorem{teo}{Theorem}[section]
\newtheorem{remark}[teo]{Remark}
\newtheorem{cor}[teo]{Corollary}
\newtheorem{lem}[teo]{Lemma}

\newtheorem{ex}[teo]{Example}

\providecommand{\F}{\ensuremath{\mathbb{F}_{q}} }

\providecommand{\gf}[1]{\ensuremath{\mathbb{F}_{#1}}}

\title{On easily computable indecomposable dimension group algebras, and group codes}

\author{ E. J. Garc\'ia-Claro\\ 
\small{Departamento de Matem\'aticas, Universidad Aut\'onoma Metropolitana-Iztapalapa}\\
\small{ C\'odigo postal 09340, Ciudad de M\'exico, M\'exico}\\
\small{eliasjaviergarcia@gmail.com} }

\begin{document}

\maketitle

\begin{abstract}
An easily computable dimension (or ECD) group code in the group algebra $\mathbb{F}_{q}G$ is an ideal of dimension less than or equal to $p=char(\mathbb{F}_{q})$ that is generated by an idempotent. This paper introduces an \textit{easily computable indecomposable dimension} (or ECID) group algebra as a finite group algebra for which all group codes generated by primitive idempotents are ECD. Several characterizations are given for these algebras. In addition, some arithmetic conditions to determine whether a group algebra is ECID are presented, in the case it is semisimple. In the non-semisimple case, these algebras have finite representation type where the Sylow $p$-subgroups of the underlying group are simple. The dimension and some lower bounds for the minimum Hamming distance of group codes in these algebras are given together with some arithmetical tests of primitivity of idempotents. Examples illustrating the main results are presented.
\end{abstract}

\textit{keywords}: group algebra; group code; $q$-orbit; $q$-cyclotomic class;  finite representation type group algebra;  ECD group code.

\section{Introduction}

Let $\gf{q}$ be the finite field with $q$ elements and $p=char(\F)$. Given $v\in \gf{q}^{n}$, its \textit{Hamming weight} $wt(v)$ is the number of non-zero entries in $v$. The \textit{Hamming distance} in $\gf{q}^{n}$ is defined as $d(v,w)=wt(v-w)$ for all $v$ and $w$. An $[n,k,d]$- \textit{linear code} $C$ over $\gf{q}$ is an $k$-dimensional subspace of $\gf{q}^{n}$ with \textit{minimum distance} (or \textit{minimum weight}) $d(C):=\min \{d(v,w): v,w\in C \wedge v\neq w\}=\min \{wt(c): c\in C-\{0\}\}$. Constructing linear codes with large \textit{information rate} $k/n$, while preserving the \textit{relative minimum distance} $d/n$ as large as possible, is one of the major challenges in coding theory as well as computing or offering bounds for the parameters of a given code by using methods that reduce the complexity of those already known. \\

Determining the dimension and minimum distance of a linear code is relevant to know how good or bad this is for error correction. For this purpose, it is common to study linear codes with a richer algebraic structure (rather than just vector spaces), such as those that are ideals of a finite-dimensional algebra. The group algebra $\F G$ of a finite group $G$ over the field $\F$ is the vector space generated by $G$ with coefficients in $\F$, i.e., $\F G:=\left\lbrace \sum_{g\in G} \lambda_{g} g \, : \, \lambda_{g}\in \F \right\rbrace$ \cite{grouprings}. With the usual sum of vectors and the multiplication given by extending  the operation of $G$, $\F G$ is an $\F$-algebra with basis $G$. A \textit{group code} $C$ in $\F G$ (or $G$-code) is a left (right) ideal of $\F G$, also named abelian code when $G$ is an abelian group. For $v\in \F G$, its \textit{Hamming weight} $wt(v)$ is the number of non-zero entries in the coordinate vector of $v$ with respect to $G$. The \textit{Hamming distance} in $\F G$ is defined as $d(v,w)=wt(v-w)$ for all $v$ and $w$. Similarly to linear codes, for a group code $C$, its \textit{minimum distance} (or \textit{minimum weight}) is defined as $d(C):=\min \{d(v,w): v,w\in C \wedge v\neq w\}$. Group codes are of interest for several reasons among which are that these are asymptotically good \cite{asym-good1, asym-good2}, that outstanding classic codes such as the extended binary Golay \cite{golay} code or the binary Reed-Muller codes \cite{reed-muller} are of this kind, and that the parameters and properties of these codes can be studied using tools of representation theory \cite{on-ckcodes, LCD1, LCD2, borello2, garcia-tapia}. A family of group codes, for which the dimension can be computed in a simple manner, is the one of \textit{easily computable dimension} (or ECD) \textit{group code}. $C \subseteq \F G$ is an ECD group code if it is generated by an idempotent such that $dim_{\F}(C)\leq p$ (see \cite[pg 13]{garcia-tapia}). The dimension of an ECD group code can be computed as follows: if $C=\F Ge$ where $e$ is an idempotent in $\F G$, $dim_{\F}(C)\leq p$, $\lambda_{1}(e)$ denotes the coefficient of $e$ in $1$, and $r$ is the least non-negative integer in the class $|G|\lambda_{1}(e)\in \gf{p}$ (this is a class modulo $p$ by \cite[Lemma 3.1]{garcia-tapia}),  by \cite[Theorem 3.1, part 3]{garcia-tapia},
\[ dim_{\F}(C) = \begin{cases} r  &\mbox{if } r\neq 0 \\
p &\mbox{if }  r=0  \end{cases}.\] 
Some works computing or presenting bounds for the parameters of groups codes of the late years are the following: in \cite{bch-di}, M. Borello et al. introduced lower bounds on the dimension and minimum weight of group codes that are principal ideals in dihedral group algebras and determined a BCH-type bound for these. In \cite{elia-gorla}, M. Elia et al. addressed the problem of determining bounds for the dimension of group codes that are principal ideals by studying the characteristic polynomial of the right/left regular representation of a generator. In \cite{garcia-tapia}, we proved several relations for the dimension of principal ideals in group algebras by examining minimal polynomials of regular representations, which allowed us to provide an improvement of a lower bound for the dimension of group codes given in \cite{elia-gorla}, among other results. Recently, in \cite{borello2} M. Borello et al. studied a new way of constructing group codes called the Schur product, and diverse properties of group codes, among which is, the following lower bound for the minimum Hamming distance of an arbitrary group code: \cite[Corollary 2.6]{borello2} if $C\subseteq \F G$ is a group code, then \[\frac{|G|}{dim_{\F}(C)}\leq d(C).\]
In this work, necessary and sufficient conditions under which all the group codes generated by primitive idempotents in a group algebra are ECD are studied. Some of these conditions are later applied to compute the dimension of group codes in these algebras and to give lower bounds for their minimum Hamming distances (via \cite[Corollary 2.6]{borello2}). The manuscript is organized as follows. In Section \ref{s2}, some properties of the size of certain orbits (known as $q$-orbits) of a finite abelian group $G$, are studied. Later, in Section \ref{s3}, splitting fields for finite semisimple commutative group algebras are characterized. In Section \ref{s4}  characterizations of ECID group algebras are given (by applying the results of Sections \ref{s2} and \ref{s3}), and these results are later applied in Section \ref{s5} to compute the dimension and to give some lower bounds for the minimum Hamming distance of group codes in these algebras. In addition,  in Section \ref{s5}, arithmetical conditions (in terms of the minimum Hamming distance) are given for determining whether an idempotent in a finite group algebra is not primitive.  Examples illustrating the main results are included.

\section{Some properties of the size of a $q$-orbit}\label{s2}

In this work, it will be used the term ``group code'' instead of ``ideal'' (to refer to an ideal of finite group algebra) mostly in a context where the Hamming distance is involved. On what remains, $\F$ is the finite field with $q$ elements, $1\neq G$ is a finite abelian group such that $p=char(\F)\nmid |G|$, unless stated otherwise. If $\zeta$ is a primitive $exp(G)$-th root of unity in some extension field of $\F$, then  $\mathcal{G}:=Gal(\F(\zeta)/\F)=\langle \alpha \rangle$ (where $\alpha$ denotes the Frobenius automorphism) acts on $G$ as $\alpha^{j} \cdot g:=g^{q^{j}}$ for $j=0,...,|\mathcal{G}|-1$ and all $g\in G$. The orbits under this action are called the $q$-orbits (also known as $q$-cyclotomic classes \cite{PIGA}) of $G$. This section is dedicated to studying some properties of the $q$-orbits that will be necessary later for characterizing the ECID group algebras.\\  
 
 From now on, $\mathcal{S}_{g}$ will denote the $q$-orbit of $g$ and $t_{g}:=\min \{a\in \mathbb{Z}_{>0}: g^{q^{a}}=g\}$ for all $g\in G$, and $l:=lcm\{t_{g}\}_{g\in G}$. Note that 
  \[ \mathcal{S}_{g}=\{\alpha^{j}\cdot g: j=0,...,|\mathcal{G}|-1\}=\{ g^{q^{j}}: j=0,...,|\mathcal{G}|-1\}=\{ g^{q^{j}}: j=1,...,t_{g}\}.
 \]
Thus $t_{g}=\min \{a\in \mathbb{Z}_{>0}: g^{q^{a}}=g\}=\min \{a\in \mathbb{Z}_{>0}: q^{a}\equiv 1\; mod(o(g))\}=|{S}_{g}|$ for all $g\in G$. The following Lemma presents some basic properties related to this number that will be of use later.
 
\begin{lem}\label{elemental}

The following statements hold:

\begin{enumerate}

\item Let $a\in \mathbb{Z}_{>0}$ and $g\in G$. Then $g^{q^{a}}=g$ $\Leftrightarrow$ $t_{g}\mid a$.
 
\item Let $a\in \mathbb{Z}_{>0}$. Then $g^{q^{a}}=g$ for all $g\in G$ $\Leftrightarrow$ $l\mid a$.

 
\item Let $g,h\in G$ such that $o(g)\mid o(h)$, then $t_{g}\mid t_{h}$.
\end{enumerate}
\end{lem} 

\begin{proof}

\begin{enumerate}

\item Suppose that there exists $g_{0}\in G$ such that $g_{0}^{q^{a}}=g_{0}$ and $t_{g_{0}}\nmid a$, then $a=t_{g_{0}}c + r$ with $0\lneq r< t_{g_{0}}$. Thus, since $g_{0}^{q^{t_{g_{0}}}}=g_{0}$, then 
\[g_{0}=g_{0}^{q^{a}}= g_{0}^{q^{t_{g_{0}}c + r}}=g_{0}^{q^{t_{g_{0}} c}q^{r}}=(\underbrace{(\cdots(g_{0}^{q^{t_{g_{0}} }})^{q^{t_{g_{0}} }}\cdots)^{q^{t_{g_{0}}}}}_{c\ times})^{q^{r}}=g_{0}^{q^{r}},\]
which contradicts the minimality of $t_{g_{0}}$. Therefore, if $g^{q^{a}}=g$ then $t_{g}\mid a$. Conversely, if $t_{g}\mid a$, $a=t_{g}c$. Hence, since $g^{q^{t_{g}}}=g$, 
\[g^{q^{a}}= g^{q^{t_{g}c}}=\underbrace{(\cdots(g^{q^{t_{g} }})^{q^{t_{g} }}\cdots)^{q^{t_{g}}}}_{c\ times}=g.\]
\item Suppose that $g^{q^{a}}=g$ for all $g\in G$, then  $t_{g}\mid a$ for all $g\in G$ (by part $1$). Thus $l=lcm\{t_{g}\}_{g\in G}\mid a$. Conversely, suppose $l\mid a$. Since $t_{g}\mid l$ for all $g\in G$, then $t_{g}\mid a$ for all $g\in G$, i.e, $g^{q^{a}}=g$ for all $g\in G$ (by part $1$).  


\item  Since $q^{t_{h}}\equiv 1\; mod(o(h))$, there exists $k\in \mathbb{Z}$ such that $q^{t_{h}}= 1+ k\cdot o(h)= 1+ k\cdot( c\cdot o(g))$ for some $c\in \mathbb{Z}$ (because $o(g)\mid o(h)$). Thus $q^{t_{h}}\equiv 1\; mod(o(g))$, implying that $g^{q^{t_{h}}}=g$ and so $t_{g}\mid t_{h}$ (by part $1$). 
\end{enumerate}
\end{proof}

 \section{Splitting fields for finite abelian groups}\label{s3}

In this section, some characterizations for the splitting fields of $\F G$ containing $\F$ are presented. The main result,  Theorem \ref{eq-split}, will be used later in the study of ECID group algebras. The ideals (modules) considered in this work will be left ideals (modules) unless stated otherwise. A simple (irreducible) module over a ring is a module with no submodules other than the zero and itself. A field  $F$ is a splitting field for $FG$ (or $G$) if $End_{FG}(S)=F$ for every simple $FG$-module $S$ \cite[p. 22]{larry}.\\ 
It is well-known that in every finite abelian group $G$ there exists an element $w \in G$ such that $o(w)=exp(G)$.

\begin{lem}\label{super0}  Let $w\in G$ be such that $o(w)=exp(G)$. Then $t_{w}=l$.
\end{lem}

\begin{proof}
 Since $o(h)\mid lcm\{o(g)\}_{g\in G}=exp(G)=o(w)$ for all $h\in G$, then $t_{h}\mid t_{w}$ (by Lemma \ref{elemental}, part $3$) for all $h\in G$, implying that $l=lcm\{t_{h}\}_{h\in G}\mid t_{w}$. On the other hand, $t_{w}\mid lcm\{t_{h}\}_{h\in G}$ so that $l=t_{w}$.
\end{proof}

\begin{teo}\label{eq-split} Let $w \in G$ be such that $o(w)=exp(G)$ and $F$ be an extension field of $\F$ of index $t$. Then the following statements are equivalent:

\begin{enumerate}

\item $F$ is a splitting field for $G$.

\item $F$ contains a primitive $exp(G)$-th root of unity.

\item $t_{w}\mid t$.

\item $exp(G)\mid q^{t}-1$.  
\end{enumerate}
\end{teo}

\begin{proof}

$1 \Rightarrow 2$. Let $\{p_{i}\}_{i=1}^{s}$  be the set of prime divisors of $|G|$, $Syl_{p_{i}}(G)$ be the set of $p_{i}$-sylow subgroups of $G$, and $g_{i}\in Syl_{p_{i}}(G)$ be an element of maximal order $o_{i}=exp(Syl_{p_{i}}(G))$ for all $i$. Let $I$ be a minimal ideal of $FG$. Consider the automorphism of $FG$-modules $l_{g_{i}}:I\rightarrow I$ given by $l_{g_{i}}(x)=g_{i}x$ for all $i$, since $F$ is a splitting field for $G$, $l_{g_{i}}(x)= \zeta_{i}x$ for some $\zeta_{i} \in F$, all $i$ and all $x\in I$. Thus $\zeta_{i}$ is a primitive $o_{i}$-th root of unity for all $i$. Finally, since $o_{1},...,o_{s}$ are pairwise relatively prime, $\zeta:=\prod_{i=1}^{s} \zeta_{i}$ has multiplicative order $\prod_{i=1}^{s}o_{i}=\prod_{i=1}^{s}exp(Syl_{p_{i}}(G))=exp(G)$, i.e, $\zeta$ is a primitive $exp(G)$-th root of unity in $F$.\\
 
$2 \Rightarrow 3$. Let $\zeta$ be a primitive $exp(G)$-th root of unity in $F$. Then the number of non-isomorphic simple modules (minimal ideals) is equal to the number of conjugation classes of $G$ (by \cite[Corollaries 24.11 and 4.4]{larry}), this is $|G|$. Thus the minimal ideals of $FG$ have dimension $1$. Hence the orbits under the action of $Gal(F(\zeta)/F)$ on $G$ (where $\zeta$ is any primitive $exp(G)$-th root of unity in some extension of $F$) have size $1$, meaning that the $q^t$-orbits have size $1$ (by \cite[Theorem 4.3, part 2]{garcia-tapia}); or equivalently, $g^{q^{t}}=g$ for all $g\in G$. Therefore $t_{w}=l\mid t$, by Lemmas \ref{super0} and \ref{elemental} (part $2$).\\

$3 \Rightarrow 4$. If $t_{w}\mid t$, then $l\mid t$ (by Lemma \ref{super0}), implying that $w^{q^{t}}=w$ (by Lemma \ref{elemental}, part $2$), i.e., $w^{q^{t}-1}=1$. Therefore $exp(G)=o(w)\mid q^{t}-1$. \\

$4 \Rightarrow 1$. Suppose $exp(G)\mid q^{t}-1$. Then $g^{q^{t}-1}=1$ for all $g\in G$, i.e., $g^{q^{t}}=g$ for all $g\in G$ and so the $q^{t}$-orbits have size $1$. Hence, the minimal ideals  of $FG$ (simple $FG$-modules) have dimension $1$ over $F$ (by \cite[Theorem 4.3, part 2]{garcia-tapia}), which implies that $End_{FG}(S)=F$ for every simple module $S$ (because $End_{FG}(S)\subseteq End_{F}(S)$ and $dim_{F}(End_{F}(S))=1$), i.e., $F$ is a splitting field for $G$.
\end{proof}

It is well-known that if $H$ is a finite group (not necessarily abelian) and $F$ is a field with $char(F)\nmid |H|$ that contains a primitive $exp(H)$-th root of unity, $F$ is a splitting field for $H$ (see, e.g., \cite[Corollary 24.11]{larry}). Even so, it seems that a result establishing the equivalences that appear in Theorem \ref{eq-split} is not known yet. This theorem does not hold true in general. For instance, $\gf{5}S_{3}$ has Wedderburn-Artin decomposition $ \gf{5} \oplus \gf{5} \oplus M_{2}(\gf{5}) $ (by a similar argument to the one in \cite[Example 3.6.12]{grouprings}). Thus $\gf{5}$ is a splitting field for $S_{3}$ (by  \cite[Theorem 3.34]{rep4}), but $\gf{5}$ does not contain an primitive $6$-th root of unity.

\section{Easily computable indecomposable dimension group algebras}\label{s4}

 A module $M\neq 0$ is called indecomposable if $M\neq M_{1}\oplus M_{2}$ for any non-zero submodules $M_{1}$ and $M_{2}$ of $M$.  Let $H \neq 1$ be a finite group. A principal indecomposable module of $\F H$ is a principal ideal generated by a primitive idempotent  \cite[Theorem 54.5]{rep3} (these are the indecomposable $\F H$-submodules of $\F H$). If $\F H$ is such that all the principal indecomposable modules are ECD, then it will be said that  $\gf{q}H$ is an \textbf{\textit{easily computable indecomposable dimension}} group algebra (or simply,  \textit{\textbf{ECID}} group algebra). If $\F H$ is semisimple, then its principal indecomposable modules are precisely its minimal ideals. Thus if $\F H$ is a semisimple \textit{\textbf{ECID}} group algebra it will be said that it is \textbf{\textit{minimal easily computable dimension}} group algebra (or simply,  \textit{\textbf{minimal ECD}} group algebra). In this section, conditions under which a group algebra is a minimal ECD and ECID are studied, and some characterizations of this properties are provided. It is recommended to be familiarized with the topics treated in \cite[Sections 2.5, 2.6]{grouprings},  and \cite[Section 54]{rep3} or \cite[Section $3$C]{rep4}.
 
\begin{teo} \label{super}
Let $w\in G$ be such that $o(w)=exp(G)$. Then 

\[t_{w}=max\{dim_{\F}(I): I \text{ is a minimal ideal of } \F G \}.\]
Moreover, $\F G$ is a minimal ECD group algebra iff $t_{w}\leq p$.  
\end{teo}

\begin{proof} Since  $\{t_{h}\}_{h\in G}=\{dim_{\F}(I): I \text{ is a minimal ideal of } \F G \}$ (by \cite[Theorem 4.3, part 2]{garcia-tapia}) and $t_{w}=l$ (by Lemma \ref{super0}), then
\[ t_{w}= max\{t_{h}\}_{h\in G} = max\{dim_{\F}(I): I \text{ is a minimal ideal of } \F G \}.\]
The rest follows from the definition of minimal ECD group algebra.
\end{proof}

\begin{cor}\label{c-teoprin}
 Let $F$ be an extension of $\F$ of index $t$ that is a splitting field for $G$. Then
 \[t\leq p \Rightarrow \gf{q}G \text{ is a minimal } ECD \text{ group algebra }.  \]
\end{cor}

\begin{proof} Suppose that $t\leq p $. Since $F$ is a splitting field for $G$, $t_{w}\mid t$ (by Theorem \ref{eq-split}) and so $t_{w}\leq p$. Hence, $\gf{q}G$ is a minimal ECD group algebra (by Theorem \ref{super}).
\end{proof}

The following example illustrates a method that uses Corollary \ref{c-teoprin} to construct infinitely many minimal ECD group algebras.

\begin{ex}\label{ex1}
Let $q=p^{\alpha}\neq 2$, $t'\leq p$, and $1\neq n\mid q^{t'}-1$ (by construction $p\nmid n$). If $G$ is any abelian group of exponent $n$ (there exists an infinite number of such groups),  $\gf{q}G$ is a minimal ECD group algebra (by Theorem \ref{eq-split} and Corollary \ref{c-teoprin}). For instance, let $p=5$, $q=p^{6}$ and $t'=4$. Since  $t'=4\leq p=5$, $q^{t'}-1=2^{5} \cdot 3^{2} \cdot 7 \cdot 13 \cdot 31 \cdot 313 \cdot 601 \cdot 390001$, and $G_{1}=C_{2}\times C_{2^{4}}\times C_{3^{2}}\times C_{3}$ and $G_{2}=C_{2^{3}}\times C_{2^{3}}\times C_{2^{4}}\times C_{3^{2}}$ are groups of exponent $144=2^{4}\cdot 3^{2}\mid q^{t'}-1$, then $\gf{q}G_{1}$ and $\gf{q}G_{2}$ are both minimal ECD group algebras (by Theorem \ref{eq-split} and Corollary \ref{c-teoprin}).
\end{ex}


The relation defined as $g\sim h$ iff $\left\langle g \right\rangle= \left\langle h \right\rangle$ for all $g,h\in G$ is an equivalence relation on $G$. Let $\mathcal{C}_{g}$ and $\mathcal{S}_{g}$ denote the equivalence class under $\sim$ and the  $q$-orbit of $g$ for all $g\in G$, respectively. Then $|\mathcal{C}_{g}|=\phi(o(g))$ and $|\mathcal{S}_{g}|=t_{g}$ for all $g\in G$ where $\phi$ denotes the Euler's totien function.








\begin{teo} \label{f1}

 If $\phi(exp(G))\leq p$, then $\gf{q}G$  is a minimal ECD group algebra. In particular, if $|G|=\prod_{i=1}^{s}p_{i}^{e_{i}}$ is the decomposition of $|G|$ into product of powers of distinct prime numbers and  $\prod_{i=1}^{s}p_{i}^{e_{i}-1}(p_{i}-1)< p$, then $\gf{q}G$  is a minimal ECD group algebra.
\end{teo}

\begin{proof}
Let $g\in G$. Since $p\nmid |G|$, $gcd(o(g), q^{j})=1$ for all $j\in \mathbb{Z}_{>0}$ and therefore $o(g^{q^{j}})=\frac{o(g)}{gcd(o(g), q^{j})}=o(g)$ for all $j\in \{1,...,t_{g}\}$. Thus two elements in the same $q$-orbit of $g$ generate the same cyclic group as $g$, i.e.,  $\mathcal{S}_{g}\subseteq \mathcal{C}_{g}$. Therefore  $|\mathcal{S}_{g}|\leq |\mathcal{C}_{g}|$ for all $g\in G$. Then
   \[1 \leq t_{w}=|\mathcal{S}_{w}|\leq |\mathcal{C}_{w}|=\phi(o(w))=\phi(exp(G)).\] 
Hence, if $\phi(exp(G))< p$, $t_{w} \leq p$ and so $\gf{q}G$ is a minimal ECD group algebra (by Theorem \ref{super}). The rest follows from the definition of the Euler's totient function.
\end{proof}




\begin{remark}\label{artin}
For the upcoming results, it will be useful to remember some basic facts and terminology of the theory of Artinian semisimple rings. Let $R$ be an Artinian semisimple ring. Then there exists a unique collection $\{B_{k}\}_{k=1}^{n}$ of simple rings called the simple components of $R$ such that $R=\oplus_{k=1}^{n}B_{k}$ (by \cite[Theorem 2.6.4, Proposition 2.6.6]{grouprings}), and $B_{k} \cong M_{n_{k}}(F_{k})$ (as rings) for some division ring $F_{k}$ for all $k$ (by \cite[Lemma 2.6.16]{grouprings}). In particular, if $R$ is a finite semisimple $\F$-algebra, then the simple components $B_{k}$'s of  $R$ are such that $B_{k} \cong M_{n_{k}}(F_{k})$ as $\F$-algebra for some field extention $F_{k}$ of $\F$  for all $k$ (by \cite[Lemma 4.1]{larry}). 
  In addition, if $H\neq 1$ is a finite group such that $\F H$ is semisimple (i.e., $p\nmid |H|$), the simple components $B_{k}$'s such that $B_{k}\cong F_{k}$ are called the commutative simple components of $\F H$ (these exist since at least one of the $n_{k}$'s is equal to $1$, because the ideal generated by $\F H(\sum_{h\in H}h)$ is a simple component isomorphic to $\F$).
 Thus, if $H$ is non-abelian, there exist $r,s\in \mathbb{Z}_{>0}$ such that  $\F H= \left( \oplus_{i=1}^{r} B_{i}\right) \oplus \left( \oplus_{j=1}^{s} B_{j} \right)$ where $B_{i}\cong F_{i}$ and $B_{j}\cong M_{n_{j}}(F_{j}) $  as $\F$-algebras where the fields $F_{i}$ and $F_{j}$ are extensions of $\F$ for $i=1,..., r$ and $j=1,..., s$,  and $n_{j}\geq 2$ for all $j=1,...,s$. 
\end{remark} 

For Theorem \ref{nabelian}, which presents a characterization of non-commutative semisimple group algebras that are minimal ECD, will be asumed the terminology and the decomposition of $\F H$ mentioned in Remark \ref{artin}.

\begin{teo}\label{nabelian} Let $H\neq 1$ be a finite non-abelian group such that $p\nmid |H|$, $H'$ be the commutator subgroup of $H$, and $\widehat{H'}:=\frac{1}{|H'|}\left(\sum_{h\in H'}h\right)$.  Let $G:=H/H'$ and $w\in G$ such that $o(w)=exp(G)$. Then: 

\begin{enumerate}
\item The largest dimension of a minimal ideal contained in $(\F H)\widehat{H'}$ is $t_{w}$. 

\item Every minimal ideal contained in $(\F H)(1-\widehat{H'})$ has dimension $n_{j}\cdot[F_{j}: \F]$ for some $j$.

\item $\F H$  is a minimal ECD group algebra iff  $t_{w}\leq p$  and  $max\{n_{j}\cdot[F_{j}: \F]\}_{j=1}^{s}\leq p$.
\end{enumerate}
\end{teo}

\begin{proof}
Let $U:=\left(\F H\right)\widehat{H'}$ and $V:=\left(\F H\right)(1-\widehat{H'})$. Then, by \cite[Proposition 3.6.11]{grouprings}, $\F H= U \oplus V$ where $U\cong \F G$ (as $\F $-algebras) is the sum of all commutative simple components of $\F H$ and $V$ is the sum of all the others. By construction of the simple components of $\F H$, every the minimal ideal is contained in a unique simple component.

\begin{enumerate}

\item It follows from the fact that $G$ is abelian, $U\cong \F G$ and Theorem \ref{super}.

\item Let $I$ be a minimal ideal contained in $V$. Then there exists $j\in \{1,...,s\}$ such that $I$ is contained in the simple component $ B_{j}\cong A_{j}:= M_{n_{j}}(F_{j})$. In addition, since all the minimal ideals of $A_{j}$ are isomorphic (as $A_{j}$-modules) to the left ideal  $A_{j}o_{11}$ generated by the primitive idempotent $o_{11}\in A_{j}$ that has $1$ (with $1\in F_{j}$) in the $11$-entry and zero in the other ones, then $dim_{\F}(I)=dim_{\F}(A_{j}o_{11})=n_{j}\cdot[F_{j}: \F]$ (where the last equality is by \cite[Theorem 3.4.9, part $\textit{iv}$]{grouprings}).

\item Suppose that $\F H$  is a minimal ECD group algebra. Let $I$ be the minimal ideal with largest dimension contained in $U$, then  $t_{w}=dim_{\F}(I)\leq p$ (by part $1$). Let $j\in \{1,...,s\}$ and $A_{j}:=M_{n_{j}}(F_{j})$. Since $A_{j}o_{11}$ (where $o_{11}$ is as in the proof of part $2$) is minimal with dimension $n_{j}\cdot[F_{j}: \F]$ (by \cite[Theorem 3.4.9, part $\textit{iv}$]{grouprings}) and $B_{j}\cong A_{j}$, then $A_{j}o_{11}$ is isomorphic to a minimal ideal of $B_{j}$ and so $n_{j}\cdot[F_{j}: \F]\leq p$. Therefore,  $t_{w}\leq p$  and  $max\{n_{j}\cdot[F_{j}: \F]\}\leq p$. Conversely, suppose that $t_{w}\leq p$  and  $max\{n_{j}\cdot[F_{j}: \F]\}_{j=1}^{s}\leq p$. If $I$ is a minimal ideal of $\F H $ contained in $U$, then $dim_{\F}(I)\leq t_{w}\leq p$ (by part $1$). On the other hand, if $I$ is a minimal ideal of $\F H $ contained in $V$, then $dim_{\F}(I)\in \{n_{j}\cdot[F_{j}: \F]\}_{j=1}^{s}$ (by part $2$) and so $dim_{\F}(I)\leq max\{n_{j}\cdot[F_{j}: \F]\}_{j=1}^{s}\leq p$. Hence, all the minimal ideals of $\F H$ have dimension less than or equal to $p$.
\end{enumerate}
\end{proof}


The same hypotheses and notation from Theorem \ref{nabelian} will be asumed for Corollaries \ref{cnabelian} and \ref{nmax}. 
 Let $A$ be a finite-dimensional algebra over a field $k$ and  $M$ a finitely generated $A$-module. If $E$ is an extension field of $k$, $M^{E}:=E\otimes_{k} M$ denotes the left $A^{E}$-module, where $A^{E}=E\otimes_{k}A$. $M$ is said to be absolutely simple if $M^{E}$ is a simple $A^{E}$-module for all extension field $E$ of $k$ \cite[Definition 3.42]{rep4}.
\begin{cor} \label{cnabelian} The following statements hold true:

\begin{enumerate}
\item If $t_{w}\leq p$ and $|H|-[H:H']\leq p$, then $\F H$ is a minimal ECD group algebra.

\item If $\phi(exp(G))\leq p$ or there exists $t\in \mathbb{Z}_{>0}$ such that $exp(G)\mid q^{t}-1$ with $t\leq p$; and $max\{n_{j}\cdot[F_{j}: \F]\}_{j=1}^{s}\leq p$, then $\F H$ is a minimal ECD group algebra.

\item If a minimal ideal $I_{j}$ contained in the simple component $B_{j}$ of $\F H$ is absolutely simple for all $j$, then $\F H$ is minimal ECD iff  $t_{w}\leq p$ and $max\{n_{j}\}_{j=1}^{s}\leq p$. In particular, if $\F$ is a splitting field for $H$, then $\F H$ is minimal ECD group algebra iff $max\{n_{j}\}_{j=1}^{s}\leq p$.
\end{enumerate}
\end{cor}

\begin{proof}

\begin{enumerate}

\item Suppose that $t_{w}\leq p$ and $|H|-[H:H']\leq p$. Then, by \cite[Proposition 3.6.11]{grouprings}, $dim_{\F}\left(\oplus_{i=1}^{r} F_{i}\right)=dim_{\F}\left(\F (H/H')\right)$. Hence,  $|H|-[H:H']=\sum_{j=1}^{s}n_{j}^{2}\cdot dim_{\F}(F_{j})=\sum_{j=1}^{s}n_{j}^{2}\cdot[F_{j}: \F]\leq p$. Therefore $max\{n_{j}\cdot[F_{j}: \F]\}_{j=1}^{s}<p$, implying then $\F H$ is a minimal ECD group algebra (by Theorem \ref{nabelian}, part $3$).

\item  If  $\phi(exp(G))\leq p$, then $\F G$ is a minimal ECD group algebra (by Theorem \ref{f1}). On the other hand, if there exists $t\in \mathbb{Z}_{>0}$ such that $exp(G)\mid q^{t}-1$ with $t\leq p$, then $\F G$ is a minimal ECD group algebra (by Theorem \ref{eq-split} and Corollary \ref{c-teoprin}). Thus, in both cases, $\F G$ is a minimal ECD group algebra, implying that $t_{w}\leq p$ (by Theorem \ref{super}). Thus, if $max\{ n_{j}\cdot[F_{j}: \F]\}_{j=1}^{s}\leq p$, $\F H$ is a minimal ECD group algebra (by Theorem \ref{nabelian}, part $3$).

\item Suppose that a minimal ideal $I_{j}$ contained in the simple component $B_{j}$ of $\F H$ is absolutely simple for all $j$.  Then $F_{j}=End_{\F H}(I_{j})=\F$ where the first equality is by construction of $F_{j}$ (see \cite[Lemma 4.1]{larry}) and the second one is by \cite[Theorem 3.43]{rep4} for all $j$. Thus $[F_{j}:\F]=1$ for all $j$ and so $max\{n_{j}\cdot [F_{j}: \F]\}_{j=1}^{s}=max\{n_{j}\}_{j=1}^{s}$. Hence the conditions $t_{w}\leq p$ and $max\{n_{j}\cdot [F_{j}: \F]\}_{j=1}^{s}\leq p$ in Theorem \ref{nabelian} (part $3$) become $t_{w}\leq p$ and $max\{n_{j}\}_{j=1}^{s}\leq p$. In particular, if $\F$ is a splitting field for $H$, the sum $U$ of the commutative simple components of $\F H$ is isomorphic to $\F G\cong \oplus_{i=1}^{r} \F $ (by \cite[Proposition 3.6.11]{grouprings}), so that $t_{w}=1$ (by Theorem \ref{nabelian}, part $1$). In addition, $F_{j}=\F$ for all $j$ (by \cite[Theorem 3.34]{rep4}). The rest follows from Theorem \ref{nabelian}, part $3$.
 \end{enumerate}
\end{proof}

\begin{cor}\label{nmax}
Let $\gamma=|H|-|H/H'|$, $b_{0}=max \left\lbrace  \left\lfloor \sqrt{\frac{\gamma}{f}} \right \rfloor \cdot f  \right\rbrace_{f=1}^{\lfloor \frac{\gamma}{4} \rfloor}$. Then the following statements hold:

\begin{enumerate}

\item $max\{n_{j}\cdot [F_{j}: \F]\}_{j=1}^{s}\leq b_{0}\leq \lfloor \frac{\gamma}{2}\rfloor$.

\item  If $t_{w}\leq p$, and $b_{0}\leq p$ or $\lfloor \frac{\gamma}{2}\rfloor \leq p$, then $\F H$ is a minimal ECD group algebra.

\item If a minimal ideal $I_{j}$ contained in the simple component $B_{j}$ of $\F H$ is absolutely simple for all $j$,  $ \lceil \sqrt{\frac{\gamma}{s}} \rceil \leq max\{n_{j}\}_{j=1}^{s} \leq  \lfloor \sqrt{\gamma} \rfloor$. In particular, if  $\lfloor \sqrt{\gamma} \rfloor \leq p$ and $t_{w}\leq p$, then $\F H$ is a minimal ECD group algebra. Moreover, if $\F$ is splitting field for $H$ and $\lfloor \sqrt{\gamma} \rfloor \leq p$, then $\F H$ is a minimal ECD group algebra.
\end{enumerate}
\end{cor}

\begin{proof}
Note that $\gamma$ is the dimension of the sum of the non-commutative simple components of $\F H$ (by \cite[Proposition 3.6.11]{grouprings}) so that $\gamma=\sum_{j=1}^{s}n_{j}^{2}\cdot [F_{j}:\F]$. 

\begin{enumerate}

\item Let $Y:=\{n\cdot f: n\in\mathbb{Z}_{\geq 2} \wedge n^{2}f\leq \gamma\}_{f=1}^{\lfloor \frac{\gamma}{4} \rfloor}$ and $j_{0}\in \{1,...,s\}$ be such that $n_{j_{0}}\cdot[F_{j_{0}}: \F]= max\{n_{j}\cdot[F_{j}: \F]\}_{j=1}^{s}$. Then $2\leq n_{j_{0}}$  and  $n_{j_{0}}^{2}\cdot[F_{j_{0}}: \F] \leq \sum_{j=1}^{s}n_{j}^{2}\cdot [F_{j}:\F]=\gamma $ so that $[F_{j_{0}}: \F]\in \{1,...,\lfloor \frac{\gamma}{4} \rfloor \}$. This implies that $max\{n_{j}\cdot[F_{j}: \F]\}_{j=1}^{s}\leq max(Y)$.  Note that $Y= \bigcup_{f=1}^{\lfloor \frac{\gamma}{4} \rfloor} Y_{f}$ where $Y_{f}=\{n \cdot f: n\in\mathbb{Z}_{\geq 2} \wedge n^{2} \cdot f\leq \gamma\}$ for $f=1,...,\lfloor \frac{\gamma}{4} \rfloor$, and that $max(Y_{f})=max(\{n \cdot f: n\in\mathbb{Z}_{\geq 2} \wedge n\leq \left\lfloor \sqrt{\frac{\gamma}{f}} \right \rfloor \})= \left\lfloor \sqrt{\frac{\gamma}{f}} \right \rfloor \cdot f$ for all $f$,  so that 
\begin{eqnarray*}
max(Y)&=& max \left(\bigcup_{f=1}^{\lfloor \frac{\gamma}{4} \rfloor} Y_{f}\right) = max \left\lbrace max(Y_{f}) \right\rbrace_{f=1}^{\lfloor \frac{\gamma}{4} \rfloor} \\
      &=& max \left\lbrace  \left\lfloor \sqrt{\frac{\gamma}{f}} \right \rfloor \cdot f \right\rbrace_{f=1}^{\lfloor \frac{\gamma}{4} \rfloor}=b_{0} \leq max \left\lbrace  \left( \sqrt{\frac{\gamma}{f}} \right) \cdot f \right\rbrace_{f=1}^{\lfloor \frac{\gamma}{4} \rfloor}\\
      &=& max \left\lbrace  \sqrt{\gamma} \cdot \left(\frac{f}{\sqrt{f}}\right): f\in[1,\frac{\gamma}{4}] \right\rbrace = \frac{\gamma}{2}
\end{eqnarray*}

where the last equality is because $\dfrac{x}{\sqrt{x}}$ is an increasing function. Therefore $max\{n_{j}\cdot[F_{j}: \F]\}_{j=1}^{s}\leq max(Y)=b_{0} \leq  \left\lfloor  \frac{\gamma}{2} \right\rfloor$.

\item  It follows from part $1$ and Theorem \ref{nabelian} (part $3$).

\item Suppose that a minimal ideal $I_{j}$ contained in the simple component $B_{j}$ of $\F H$ is absolutely simple for all $j$. Then $max\{n_{j}\cdot [F_{j}: \F]\}_{j=1}^{s}=max\{n_{j}\}_{j=1}^{s}$ (by a similar argument to the one given in the proof of Corollary \ref{cnabelian}, part $3$). Thus, if $j_{0}\in \{1,...,s\}$ is such that $n_{j_{0}}=max\{n_{j}\}_{j=1}^{s}$, then $n_{j_{0}}^{2}\leq \gamma$ and $\gamma \leq s \cdot  n_{j_{0}}^{2}$ so that $ \lceil \sqrt{\frac{\gamma}{s}} \rceil \leq n_{j_{0}} \leq \lfloor \sqrt{\gamma} \rfloor$. In particular, if $\lfloor \sqrt{\gamma} \rfloor\leq p$ and $t_{w}\leq p$, then $\F H$ is a minimal ECD group algebra (by Theorem \ref{nabelian}, part $3$). Moreover, if $\F$ is splitting field for $H$, then $t_{w}=1$ and every simple module is absolutely simple (by \cite[Theorem 3.34]{rep4}), therefore $\lfloor \sqrt{\gamma} \rfloor\leq p$ implies that $\F H$ is a minimal ECD group algebra.
\end{enumerate}
\end{proof}

Corollary \ref{nmax} provides an efficient way for determining whether $\F H$ is minimal ECD group algebra because the conditions necessary for its application are purely arithmetic and do not depend on knowing the Wedderburn-Artin decomposition of $\F H$. Example \ref{exnabelian} illustrates this Corollary, but first, it will be useful to consider the following Remark \ref{ceq-split}, which presents a necessary condition for being an splitting field for a finite group.

\begin{remark}\label{ceq-split}
Let $H\neq 1$ be a finite group such that $p\nmid |H|$, and  $H'$ be the commutator subgroup of $H$. If $\F$ is a splitting field for $H$, then $\F$ contains a primitive $exp(H/H')$-th root of unity.
\end{remark}

\begin{proof}
If $\F$ is a splitting for $H$, then the sum of all the commutative simple components of $\F H$ is isomorphic to both  $\oplus_{i=1}^{t} \F$ and $\F (H/H')$ (by \cite[Proposition 3.6.11]{grouprings}). Thus, by  \cite[Theorem 3.34]{rep4}, $\F$ is splitting field for $H/H'$ which is abelian so that  $\F$ contains a primitive $exp(H/H')$-th root of unity (by Theorem \ref{eq-split}).
\end{proof}

\begin{ex}\label{exnabelian}
Let $H=SL(2,3)$. Then $|H|=24=2^{3}\times 3$  and $|H/H'|=|SL(2,3)/$  $ Q|=3$ (where $Q$ denotes the quaternion group) so that $\gamma=|H|-|H/H'|=21$ and $b_{0}=max \left\lbrace  \left\lfloor \sqrt{\frac{\gamma}{f}} \right \rfloor \cdot f  \right\rbrace_{f=1}^{5}=10$. Thus, if $p\in \{11,13,17,19\}$  then $\mathbb{F}_{p}$  is a finite field such that $b_{0}\leq p$ and so  $\mathbb{F}_{p}H$ is a minimal ECD group algebra (by Corollary \ref{nmax}, part $2$). On the other hand, if $p\in \{5,7\}$, then $\lfloor \sqrt{\gamma} \rfloor=4\leq p$ so that there exists a finite extension $\mathbb{F}_{p^{a_{p}}}$ of $\mathbb{F}_{p}$ that is a splitting field for $H$, implying that $\mathbb{F}_{p^{a_{p}}}H$ is a minimal ECD group algebra (by Corollary \ref{nmax}, part $3$). For instance,  since $exp(H)=12\mid 5^{2}-1 $,  $\mathbb{F}_{5^{2}}$ contains a primitive $exp(H)$-th root of unity. Thus $\mathbb{F}_{5^{2}}$ is an splitting field for $H$  (by \cite[Corollary 24.11]{larry}), but  $\mathbb{F}_{5}$ is not splitting field for $H$  (by Corollary \ref{ceq-split}). Finally, if $p$ is a prime greater than $19$, then any proper ideal of $\mathbb{F}_{p}H$ is (trivially) ECD because $|H|-1=23\leq p$.\\

Consider the Mathieu group $M_{12}$. Since $M_{12}$ is a simple group, $\gamma=|M_{12}|- |\frac{M_{12}}{M_{12}^{'}}|= |M_{12}|- 1= 95039$. In addition, it is well-known that this group contains $15$ conjugacy classes, so that the number of non-commutative simple components is $s=14$, over any splitting field. Thus if $p=char(\F)\nmid |M_{12}|=95040=2^{6}\times 3^{3} \times 5\times 11$ and $\mathbb{F}_{q}$ is an splitting field for  $M_{12}$, then $ \lceil \sqrt{\frac{\gamma}{s}} \rceil = 83 \leq max\{n_{j}\}_{j=1}^{s} \leq  \lfloor \sqrt{\gamma} \rfloor= 308$ (by Corollary \ref{nmax}, part $3$), and so if $p\in [308, |M_{12}|-1] $ (these are $9098$ primes), then $\mathbb{F}_{q} G$ is a minimal ECD group algebra (by Corollary \ref{cnabelian}, part $3$). On the contrary, if $p\in [2, 79]-\{2,3,5,11\}$ (these are $22$ primes), then  $\mathbb{F}_{q} G$ is not a minimal ECD group algebra (by Corollary \ref{cnabelian}, part $3$). 
\end{ex}

If $F$ is a field and $H$ is an finite group such that $p\mid |H|$, it is said that $F H$ has finite representation type (FRT) if there exist only a finite number of non-isomorphic indecomposable $FH$-modules \cite[p.  431]{rep3}.

\begin{teo}\label{nsemi}
Let $H$ be a finite group with $p\mid |H|$. If $\F H$ has a principal indecomposable module $I$ with dimension less than or equal to $p$, then $dim_{\F}(I)=p$, any Sylow $p$-subgroup of $H$ is isomorphic to $C_{p}$ and $\F H$ has FRT. In particular, if $\F H$ is an ECID group algebra this holds.
\end{teo}

\begin{proof}

Let $I$ be a principal indecomposable module. Let $|H|_{p}$ denote the highest power of $p$ dividing $|H|$. Then, by \cite[Corollary 7.16, Chap. VII]{fin-gp}, $1\neq|H|_{p}\mid dim_{\F}(I)\leq p$, impliying that $|H|_{p}=dim_{\F}(I)=p$. Hence any Sylow $p$-subgroup of $H$ is isomorphic to $C_{p}$ and so $\F H$ has FRT (by \cite[Theorem 64.1]{rep3}). The rest is clear.
\end{proof}
A consequence of Theorem \ref{nsemi} is that any non-semisimple group algebra that is ECID, has FRT and the Sylow $p$-subgroups of its underlying group are isomorphic to $C_{p}$. The converse of this statement is not true in general (see Example \ref{non-semi-ECID}).\\

Let $A$ be a finite-dimensional algebra over a field. Let $\overline{A}:=A/J(A)$ and $\overline{x}:=x+ J(A)$ for all $x\in A$ (where $J(A)$ denotes the Jacobson radical of $A$). Then $\overline{A}$ is semisimple (by \cite[Proposition 5.19]{rep4}) and if $f \in  \overline{A}$  is a primitive idempotent, there exists a primitive idempotent  $e\in A$ such that $\overline{e}=f$ (by \cite[Theorem 4.9]{nagao}). In particular, if $H\neq 1$ is a finite group with $p\mid |H|$ and $A=\F H$, then $\overline{A}$ has a decomposition as the mentioned in Remark \ref{artin}. 

\begin{teo}\label{heavy}
Let $H\neq 1$ be a finite group with $p\mid |H|$. Let $A=\F H$ and $\overline{A}=A/J(A)$. Let $\overline{A}=\oplus_{j=1}^{n}B_{j}$ be the decomposition into simple components of $\overline{A}$ 
where  $ B_{j}\cong M_{n_{j}}(F_{j})$ (as $\F$-algebra). Let $f_{j}\in B_{j}$ and $e_{j}\in A$ be primitive idempotents such that $\overline{e_{j}}=f_{j}$ for all $j$. Then the following statements are equivalent:

\begin{enumerate}

\item $A$ is an ECID group algebra.

\item Every primitive idempotent in $A$ generates an ideal of dimension $p$.

\item $dim_{\F}(\overline{A}f_{j})=p-dim_{\F}(J(A)e_{j})$ for all $j$.

\item $dim_{\F}(J(A)e_{j})=p-n_{j}\cdot[F_{j}:\mathbb{F}_{q}]$ for all $j$.
\end{enumerate}
\end{teo}

\begin{proof}

 $1\Leftrightarrow 2$ follows from \cite[Theorem 54.5]{rep3} and Theorem \ref{nsemi}.\\
 
 Note that $Ae_{j}$ is a principal indecomposable module  with unique maximal submodule $J(A)e_{j}$ for all $j$   (by \cite[Theorems 54.5 and 54.11]{rep3}). Thus, since $J(A)e_{j}=Ae_{j}\cap J(A)$ for all $j$,  $Ae_{j}/J(A)e_{j}=Ae_{j}/(Ae_{j}\cap J(A))\cong_{A} (Ae_{j}+J(A))/J(A) =\overline {A}f_{j}$ for all $j$. Hence $dim_{\F}(\overline {A}f_{j})=dim_{\F}(Ae_{j})-dim_{\F}(J(A)e_{j})$ for all $j$. \\
 
$2 \Rightarrow 3$. If $A$ is an ECID group algebra, then $dim_{\F}(Ae_{j})=p$ for all $j$ and so $dim_{\F}(\overline {A}f_{j})=p-dim_{\F}(J(A)e_{j})$ for all $j$. $3 \Rightarrow 2$. If $dim_{\F}(\overline{A}f_{j})=p-dim_{\F}(J(A)e_{j})$, then  $dim_{\F}(Ae_{j})=p$ for all $j$. Thus, as $\{Ae_{j}/J(A)e_{j}\}_{j=1}^{n}$ is a set of representatives of isomorphism classes of simple $A$-modules,  $\{Ae_{j}\}_{j=1}^{n}$ is a set of representatives of isomorphism classes of principal indecomposable $A$-modules (by \cite[Corollary 54.14]{rep3}). Hence any principal indecomposable module of $A$ has dimension equal to $p$.\\

Since $\overline{A}f_{j}$ is a minimal ideal of $B_{j}\cong M_{n_{j}}(F_{j})$, then  $dim_{\F}(\overline{A}f_{j})=n_{j}\cdot[F_{j}:\mathbb{F}_{q}]$ for all $j$. Thus, the equivalence between $3$ and $4$ is clear.
\end{proof}















\begin{ex}\label{non-semi-ECID}
Let $H=\langle u,v : u^3=v^2=(uv)^3=1 \rangle=\{1,u,u^2v,v,u^2vu,\\* u^2,vu,uv, uvu, vuv, vu^2, uvu^2\}$ be the alternating group of degree $4$. Since $|H|=2^{2}\cdot 3$, the Sylow $3$-subgroups of $H$ are isomorphic to $C_{3}$ so that $\gf{3}H$ satisfies the necessary condition stated in Theorem \ref{nsemi} for being a ECID group algebra. In SageMath one can compute the matrices of the left regular representations of any element in $\gf{3}H$  with respect to the basis $H$ (with the order given above) and verify that there are $118$ idempotents $e$ such that $dim_{\gf{3}}(\gf{3}He)=3$. The coordinate vectors of these idempotents with respect to $H$ (with the order given above) are:
{\footnotesize \begin{eqnarray*}
112201020000, 101022120000, 111111111000, 121200102000, 110021202000, 102200021100,\\  100022021100, 112021000200, 101202100200, 120200012010, 100020212010, 111111100110,\\
111110011110, 100111111110, 102020010210, 100202010210, 121202112210, 112021212210,\\
121020100020, 110201200020, 120020001120,
100200201120, 121022121120, 112201221120,\\
100110000001, 100220000001, 120220120101,
110220220101, 111011011101, 101101111101,\\
120220110201, 110220210201, 112221120201, 
111222120201, 111010111011, 110101111011,\\
102221002011, 101222002011, 111101010111,
101011110111, 111100101111, 110011101111,\\
102222021111, 120220212111, 121222122111,
112221222111, 102222011211, 122221112211,\\
111222212211, 102221001021, 101222001021,
121221102021, 111221202021, 120220211121,\\
122221121121, 111222221121, 121222111221,
112221211221, 100210000002, 100120000002,\\
110001022002, 101000122002, 120210210102, 
120120210102, 120120120102, 110210220102,\\
122000001102, 100002201102, 122210121102,
110122221102, 121210122102, 110121222102,\\
120210110202, 110120210202, 110210120202,
110120120202, 122000010012, 100002210012,\\
102212001012, 102122001012, 102121002012,
101212002012, 122211012012, 101122212012,\\
112211022012, 101122122012, 122002211112,
122122121112, 112212221112, 122121212112,\\
121212212112, 122120110212, 110212210212, 
122212111212, 112122211212, 112121122212,\\
111212122212, 102211001022, 101122001022, 
101211002022, 101121002022, 122121001122,\\
101212201122, 122211211122, 121122211122,
121121122122, 111211222122, 110001000222,\\
101000100222, 121120110222, 110211210222,
112121001222, 101212101222, 112211121222,\\
111122121222, 121211112222, 111121212222,
111001122222.\hspace{4.6cm}
\end{eqnarray*}}

Since $3$ is the minimum possible dimension that an ideal generated by an idempotent might have (by \cite[Corollary 7.16, Chap. VII]{fin-gp}), these idempotents are primitive.  In addition, by using SageMath it can be seen that any other idempotent (distinct from $1$) is the sum of two or three of these $118$ idempotents, and so it is not primitive. Therefore $\gf{3}H$ is an ECID group algebra (by Theorem \ref{heavy}).\\

If $C_6$ is the cyclic group of order $6$ with generator $x$. Since $|C_{6}|=2\cdot 3$, the Sylow $2$-subgroup of $C_{6}$ is isomorphic to $C_{2}$ so that $\gf{2}C_{6}$ satisfies the necessary condition stated in Theorem \ref{nsemi} for being a ECID group algebra. However, by using SageMath, one can see that $\gf{2}C_6$ has only two primitive idempotents, which are $e_1 = 1 + x^2 + x^4$ and $e_2 = x^2 + x^4$, and that 
 $dim ((\gf{2}C_6)e_2) = 4 > 2 = char(\gf{2})$. Therefore,  $\gf{2}C_6$ is not an
ECID group algebra (by Theorem \ref{heavy}).
\end{ex}

\section{ Some lower bounds for the minimum distance of group codes}\label{s5}

In this section, some lower bounds for the minimum Hamming distance of group codes in  group algebras (including minimal ECD and ECID) and some arithmetical tests of non-primitivity of idempotents are presented.\\

Let $x\in \gf{p}$. Then
\[\mathfrak{D}(x) := \begin{cases} r  &\mbox{if } r\neq 0 \\
p &\mbox{if }  r=0  \end{cases}\]
where $r$ is the least non-negative integer in $x$.

\begin{teo}\label{hd-bounds}  Let $G$ be a finite abelian group such that $p\nmid |G|$, $C\subseteq \F G$ be a minimal abelian code and $e\in \F G$ be a primitive idempotent such that $C=\F Ge$. Let $x=|G|\lambda_{1}(e)$ where $\lambda_{1}(e)$ is the coefficient of $e$ in $1$ (x is a class modulo $p$ by \cite[Lemma 3.1]{garcia-tapia}). Let $w\in G$ be an element such that $o(w)=exp(G)$ and $\{p_{i}\}_{i=1}^{s}$ be the set of prime divisors of $|G|$. Then

\begin{enumerate}

\item  \[\prod_{i=1}^{s}\dfrac{p_{i}}{p_{i}-1}\leq \dfrac{|G|}{t_{w}}\leq d(C).\]
In particular, if  $f\in \F G$ is an idempotent and $C'=\F Gf$ is  such that  $d(C')< \dfrac{|G|}{t_{w}} $, then  $f$ is not primitive. 

\item  If  

\begin{enumerate}[(a)]
\item $t_{w}\leq p$,

\item or $\exists t \in \mathbb{Z}_{>0}$ such that $exp(G)\mid q^{t}-1$ and $t\leq p$,

\item or $\phi(exp(G))\leq p$ (where $\phi$ denotes the Euler's totient function),
\end{enumerate}

 then $dim_{\F}(C)=\mathfrak{D}(x)$ and

\[\prod_{i=1}^{s}\dfrac{p_{i}}{p_{i}-1} \leq \dfrac{|G|}{\mathfrak{D}(x)}\leq d(C).\]
In particular, if $(c)$ is fulfilled,
\[\dfrac{|G|}{p} \leq \prod_{i=1}^{s}\dfrac{p_{i}}{p_{i}-1} \leq \dfrac{|G|}{\mathfrak{D}(x)}\leq d(C)\]
\end{enumerate}
\end{teo}

\begin{proof}

\begin{enumerate}

\item Since  $\mathcal{S}_{w}\subseteq \mathcal{C}_{w}$, then $t_{w}=|\mathcal{S}_{w}|\leq |\mathcal{C}_{w}|=\phi(exp(G))$. Hence $dim_{\F}(C)\leq t_{w}\leq \phi(exp(G))$ (by Theorem \ref{super}), implying that
 \[\frac{|G|}{\phi(exp(G))}=\prod_{i=1}^{s}\dfrac{p_{i}}{p_{i}-1}\leq \frac{|G|}{t_{w}} \leq \frac{|G|}{dim_{\F}(C)}\leq d(C)\]
 
 where the last inequality is by  \cite[Corollary 2.6]{borello2}.
 
 \item Suppose $(a)$, then $\F G$ is a minimal ECD group algebra (by Theorem \ref{super}). Suppose $(b)$, then $\F G$ is a minimal ECD group algebra (by Theorem \ref{eq-split} and Corollary \ref{c-teoprin}). Suppose $(c)$, then $\F G$ is a minimal ECD group algebra (by Theorem \ref{f1}). Thus, if any of the conditions $(a)-(c)$ is fulfilled, $\F G$ is a minimal ECD group algebra and so $dim_{\F}(C)=\mathfrak{D}(x)$ (by \cite[Theorem 3.1, part 3]{garcia-tapia}). Hence, by part $1$,
\[\frac{|G|}{\phi(exp(G))}=\prod_{i=1}^{s}\dfrac{p_{i}}{p_{i}-1} \leq \dfrac{|G|}{\mathfrak{D}(x)}\leq d(C).\]

In particular, if $(c)$ is fulfilled, $dim_{\F}(C)=\mathfrak{D}(x)\leq t_{w}\leq \phi(exp(G))\leq  p$ so that
  \[\dfrac{|G|}{p} \leq \frac{|G|}{\phi(exp(G))}= \prod_{i=1}^{s}\dfrac{p_{i}}{p_{i}-1} \leq \dfrac{|G|}{\mathfrak{D}(x)}\leq d(C).\]
\end{enumerate}
\end{proof}

Although in some cases the bound $\dfrac{|G|}{dim_{\F}(C)}\leq d(C)$ could be reached (by \cite[Theorem 2.10]{borello2}), generally the bounds in Theorem \ref{hd-bounds} might be far from the minimum distance. Thus their importance lies in the the meager cost required to get them.

\begin{teo}\label{final}
Let $H\neq 1$ be a finite group. Let $0\neq C\subseteq \F H$ be a group code and $1\neq e\in \F H$ be an idempotent such that $C=\F He$.

\begin{enumerate}

\item Let $r$ be the least non-negative integer in the class $|H|\lambda_{1}(e)$ (which is in $\mathbb{F}_{p}$). Then there exists $k\in \left\lbrace 0,..., \left\lfloor \dfrac{|H|-(r+1)}{p} \right \rfloor  \right\rbrace$ such that
 \[dim_{\F}(C)=r+kp \, \text{ and }  \, \frac{|H|}{r+kp}\leq d(C).\]

 \item Let $\F H$ be semisimple with Wedderburn-Artin decomposition $\F H\cong \left( \oplus_{i=1}^{r} F_{i} \right) \oplus \left( \oplus_{j=1}^{s} M_{n_{j}}(F_{j})\right)$ where $n_{j}\geq 2$ for $j=1,...,s$. Let  $w\in H/H'$ be such that $o(w)=exp(H/H')$ and $a\in \mathbb{Z}_{>0}$ be such that $max\{t_{w}, max\{n_{j}\cdot[F_{j}: \F]\}_{j=1}^{s}\}\leq a$. If
  \[d(C)<\dfrac{|H|}{a},\]
 then $e$ is not primitive. In particular, if $\F$ is an splitting field for $H$ and $d(C)<\frac{|H|}{max\{n_{j}\}_{j=1}^{s}}$, then $e$ is not primitive.
 
\item Let $\{e_{i}\}_{i=1}^{m}$ be a set of  primitive idempotents that are pairwise orthogonal with $e=\sum_{i=1}^{m}e_{i}$, $r_{i}$  be the least non-negative integer in the class $|H|\lambda_{1}(e_{i})$, $J=\{i\in\{1,...,m\}: r_{i}= 0\}$, and
 $\F H$ be an ECID group algebra. Then \[dim_{\F}(C)=\left(\sum_{i\notin J}r_{i}\right)+|J|\cdot p \, \text{ and } \,\dfrac{|H|}{\left(\sum_{i\notin J}r_{i}\right)+|J|\cdot p}\leq d(C).\]
 
\item Let $\F H$ be an ECID group algebra and $d(C)<\dfrac{|H|}{p}$. Then $e$ is not primitive.
\end{enumerate} 
\end{teo}

\begin{proof}

\begin{enumerate}

\item By \cite[Theorem 3.1, part 3]{garcia-tapia}, $dim_{\F}(C)$ is congruent with $r$ modulo $p$, i.e., there exists $k\in \mathbb{Z}_{\geq 0}$ such that $dim_{\F}(C)=r+kp$. On the other hand, since  $C\subsetneq \F H$, $dim_{\F}(C)=r+kp \leq |H|-1$, i.e., $k\leq \frac{|H|-(r+1)}{p}$. Therefore $k\in \left\lbrace 0,..., \left\lfloor \frac{|H|-(r+1)}{p} \right \rfloor\right\rbrace$. The rest follows from \cite[Corollary 2.6]{borello2}.

\item Suppose  $e$ is a primitive idempotent, then $dim_{\F}(C)\leq max\{t_{w}, max\{n_{j}\cdot[F_{j}: \F]\}_{j=1}^{s}\}\leq a$ (by Theorem \ref{nabelian} parts $1$ and $2$). Thus, by \cite[Corollary 2.6]{borello2}, $\dfrac{|H|}{a}\leq d(C)$. The rest follows from the fact that if $\F$ is an splitting field for $G$, $max\{t_{w}, max\{n_{j}\cdot[F_{j}: \F]\}_{j=1}^{s}\}=max\{n_{j}\}_{j=1}^{s}$.

\item Since $e=\sum_{i=1}^{m}e_{i}$, $C=\oplus_{i=1}^{m}\F He_{i}$ and so \[dim_{\F}(C)=\sum_{i=1}^{m}dim_{\F}(\F He_{i})=\sum_{i=1}^{m}\mathfrak{D}(|H|\lambda_{1}(e_{i}))= \left(\sum_{i\notin J}r_{i} \right)+|J|\cdot p\]  where the second equality is because $\F H$ is an ECID group algebra. The rest follows from \cite[Corollary 2.6]{borello2}.

\item Suppose $e$ is primitive, then $dim_{\F}(C)=p$ (by Theorem \ref{heavy}) and so $\dfrac{|H|}{p}\leq d(C)$ (by \cite[Corollary 2.6]{borello2}).
\end{enumerate}
\end{proof}
Since every idempotent is the sum of primitive orthogonal ones, the characterizations of finite ECID  group algebras given Section \ref{s4} may be used to construct these algebras, and to compute the dimension and some lower bounds for the minimum distance of group codes generated by (non-trivial) idempotents in them (via Theorem \ref{final}, part $3$).
\begin{ex}\label{exfinal}

 Note that $SL(2,3)$ is isomorphic to the subgroup of $S_{8}$ obtained in SageMath as $``SL(2,3).as\_permutation\_group(\,)"$ given by
\begin{eqnarray*}
H:=\langle (235)(678), (1247)(3685) \rangle= \{1_{H},(1247)(3685), (235)(678), (253)(687),\\
                                           (14)(27)(38)(56), (137)(248), (126)(475), (14)(285736), (1742)(3586),\\
                                           (123478)(56), (152467)(38), (168453)(27), (135486)(27), (187432)(56),\\
                                            (176425)(38), (1546)(2873), (14)(263758), (162)(457), (173)(284),\\
                                            (1348)(2576), (1843)(2675), (185)(364), (158)(346),(1645)(2378)\}.
\end{eqnarray*}
By Example \ref{exnabelian},  $\mathbb{F}_{25}$ is an splitting field for $H$ and $|H/H'|=3$. As $H$ has $7$-conjugacy classes,  $\mathbb{F}_{25} H\cong ( \mathbb{F}_{25} \oplus \mathbb{F}_{25}\oplus \mathbb{F}_{25}) \oplus \left( \oplus_{j=1}^{4} M_{n_{j}}(\mathbb{F}_{25})\right)$. An easy arithmetical computation leads to that $n_{1}=n_{2}=n_{3}=2$ and $n_{4}=3$, i.e., \[\mathbb{F}_{25} H\cong  \mathbb{F}_{25} \oplus \mathbb{F}_{25}\oplus \mathbb{F}_{25}  \oplus  M_{2}(\mathbb{F}_{25})\oplus M_{2}(\mathbb{F}_{25})\oplus M_{2}(\mathbb{F}_{25}) \oplus M_{3}(\mathbb{F}_{25}).\]
Thus $\mathbb{F}_{25} H$ is a minimal ECD group algebra (by Corollary \ref{cnabelian}, part $3$). Note that, the idempotent $e:=1-\widehat{H'}$  with $\widehat{H'}:=\frac{1}{|H'|}\left(\sum_{h\in H'}h\right)$ is such that $\mathbb{F}_{25} He$ is the sum of the non-commutative simple components of $\mathbb{F}_{25} H$ (by \cite[Proposition 3.6.11]{grouprings}), i.e., $\mathbb{F}_{25} He\cong  M_{2}(\mathbb{F}_{25})\oplus M_{2}(\mathbb{F}_{25})\oplus M_{2}(\mathbb{F}_{25}) \oplus M_{3}(\mathbb{F}_{25}) $.  So that the minimal ideals contained in $\mathbb{F}_{25} He$ are exactly those with dimension equal to $2$ or $3$.\\

 Let $\gamma$ be a root of the polynomial $p(x) = x^2+4x+
2 \in \mathbb{F}_{5}[x]$ in some extension field of $\mathbb{F}_{5}$. Since $p(x)$ is irreducible $\mathbb{F}_{25}= \mathbb{F}_{5}(\gamma)$. Let $e_{1}=\frac{1}{|H|}\left(\sum_{h\in H}h\right)$, and $e_{2}$ and $e_{3}$ the elements with coordinate vectors  (with the order given above over $H$) $(3, \gamma + 1, 3, 3, 2, 3\gamma + 4, \gamma + 1, 2, 4\gamma + 4, \gamma + 1, 4\gamma + 1, \gamma + 4, 3\gamma + 4, 2\gamma + 1, 4\gamma + 4, 
                 4\gamma + 1, 2, \gamma + 4, 4\gamma + 4, 3\gamma + 4, 2\gamma + 1, 2\gamma + 1, 4\gamma + 1, \gamma + 4)$ and $(2, 1, 2, 2, 2, 1, 1, 2, 1, 1, 1, 1, 1, 1, 1, 1, 2, 1, 1, 1, 1, 1, 1, 1)$, respectively.  By computing the matrices of the left regular representations of $e_{2}$ and $e_{3}$ in SageMath one can check that these are generating idempotents of ideals of dimensions $2$ and $3$ (see Table \ref{tab2}), respectively. In addition, since the product of these idempotents with $\widehat{H'}=1-e$ is zero, they are contained in $\mathbb{F}_{25} He$ and so must be primitive (because they have the adequate dimensions). Thus the information contained in Tables \ref{tab1} and \ref{tab2} is obtained:

\begin{table}[h]
\begin{center}
\begin{tabular}{|c|c|c|c|}
\hline 
$i$ &  $\lambda_{1}(e_{i})$ & $|H|\lambda_{1}(e_{i})$&  $\mathfrak{D}(|H|\lambda_{1}(e_{i}))$ \\
\hline  \hline
$1$ & $4$ & $24\cdot 4=96$&  \textcolor{red}{$1$}\\
\hline
$2$ & $3$ & $24\cdot 3=72$ &  \textcolor{red}{$2$}\\
			\hline
$3$ & $2$ & $24\cdot 2=48$ &  \textcolor{red}{$3$}\\
\hline
\end{tabular}
\end{center}
\caption{}\label{tab1}
\end{table}


\begin{table}[h]
\begin{center}
\begin{tabular}{|c|c|c|c|}
\hline 
$I$ & $dim(\oplus_{i\in I}\mathbb{F}_{25} He_{i})$ & $|H|/dim(\oplus_{i\in I}\mathbb{F}_{25} He_{i})$ & $d(\oplus_{i\in I}\mathbb{F}_{25} He_{i})$  \\
\hline  
$\{1\}$ & \textcolor{red}{$1$} & $24$& $24$\\
\hline
$\{2\}$ & \textcolor{red}{$2$} & $12$ & $18$\\
\hline
$\{3\}$ & \textcolor{red}{$3$} & $8$ & $12$ \\
\hline
$\{1,2\}$ & $3$ & $8$& $15$\\
\hline
$\{1,3\}$ & $4$ & $6$ & $6$ \\
\hline
$\{2,3\}$ & $5$ & $8$ & $9$ \\
\hline
$\{1,2,3\}$ & $6$ & $4$ & $6$ \\
\hline
\end{tabular}
\end{center}
\caption{The first three entries in the second column were obtained applying \cite[Lemma 2.3]{garcia-tapia}. These coincide with the elements in the last column of Table \ref{tab1} because $\mathbb{F}_{25} H$ is a minimal ECD group algebra. All the remaining dimensions can be easily computed as the sums of these, and one gets the lower bounds for the minimum distances that appear in the third column. This illustrates Theorem \ref{final} (part $3$).}\label{tab2}
\end{table}
\end{ex}

\section*{Conclusion}
In this work, ECID group algebras, and group codes in such algebras are studied. These results allow us to determine relations for the parameters of group codes generated by non-trivial idempotents. Moreover, certain criteria are presented to ascertain whether an idempotent in a group algebra is not primitive.\\

Since any group code (ideal) generated by an idempotent is a direct sum of principal indecomposable ones,  in an ECID group algebras the dimension of such ideals (which are all, in the semisimple case) can be easily computed as the sum of the dimension of the indecomposable ones, which provides lower bounds for their minimum distances, as was shown on Example \ref{exfinal}.\\ 

The theory presented herein may be of use to widen the understanding of the dimension of ideals in group algebras and the minimum Hamming distance of group codes or to construct ECID group algebras, and group codes with desired parameters. 


\begin{thebibliography}{}




\bibitem[1]{asym-good1}L. M. Bazzi and S. K. Mitter. Some randomized code constructions from group actions. IEEE Trans. Inform. Theory, vol. 52, no. 7, pp. 3210-3219, 2006.


\bibitem[2]{reed-muller} S. D. Berman. On the theory of group codes. Cybernetics, vol. 3, no. 1, pp.25-31, 1969.


\bibitem[3]{golay} F. Bernhardt, P. Landrock, and O. Manz. The extended golay codes considered as ideals. J. Combin. Theory Ser. A, vol. 55, no. 2, pp. 235-246, 1990.


\bibitem[4]{LCD1} M. Borello, J. De La Cruz, and W. Willems. On group codes with complementary duals. Des. Codes Cryptogr., vol. 86, pp. 2065-2073, 2018.

\bibitem[5]{on-ckcodes} M. Borello, J. De La Cruz, and W. Willems. On checkable codes in group algebras. J. Algebra Appl., vol. 21, no. 6, paper no. 2250125, 2022.

 

\bibitem[6]{bch-di} M. Borello, A. Jamous. Dihedral codes with prescribed minimum
distance. In Bajard, J.C., Topuzo\u{g}lu, A. (eds) Arithmetic of Finite Fields. WAIFI 2020. Lecture Notes in Computer Science,  Springer, Cham, vol. 12542, 2021. 

\bibitem[7]{asym-good2}M. Borello and W. Willems. Group codes over fields are asymptotically good. Finite Fields Appl., vol.68, paper no. 101738, 2020.

\bibitem[8]{borello2} M. Borello, W. Willems, G. Zini, On ideals in group algebras: An uncertainty principle and the Schur product. Forum Mathematicum, 2022-0064. 


\bibitem[9]{rep3} C. W. Curtis, I. Reiner: \textit{Representation theory of finite groups and associative algebras}, John Wiley and Sons, 1962.


\bibitem[10]{rep4} C. W. Curtis, I. Reiner: \textit{Methods of representation theory with applications to finite groups and orders}, vol I. John Wiley and Sons, 1981.


\bibitem[11]{split} C. Ding, D. R. Kohel, S. Ling, Split group codes. IEEE Trans. Inform. Theory, vol. 46, no. 2, pp. 485-495, 2000.

\bibitem[12]{LCD2} S. Dougherty, J. Gildea, A. Korbanc, and A. Roberts. Group LCD and group reversible LCD codes. Finite Fields Appl., vol. 83, paper no. 102079, 2022.

\bibitem[13]{larry} L. Dornhoff: \textit{Group Representation Theory}, Part $A$. Dekker, New York, 1971.


\bibitem[14]{elia-gorla} M. Elia, E. Gorla, Computing the dimension of ideals in group algebras, with an application to coding theory, JP Journal of Algebra, Number Theory and Applications, vol. 45, no. 1, pp. 13-28, 2020.

\bibitem[15]{n.simpcomp} R. Ferraz,  Simple components and central units in group algebras. J. Algebra, vol. 279, pp. 191-203.


\bibitem[16]{garcia-tapia} E. J. Garc\'ia-Claro, H. Tapia-Recillas,
Tapia-Recillas, On the dimension of ideals in group algebras, and group codes,   J. Algebra Appl., vol. 21, no. 2, paper no. 2250024, 2022. 


\bibitem[17]{fin-gp} B. Huppert ,  N. Blackburn: \textit{Finite Groups II}. Springer, Berlin, 1982.





\bibitem[18]{PIGA} S. Jitman, S. Ling, H. Liu, and X. Xie, Abelian codes in principal ideal group algebras, IEEE Trans. Inform. Theory, vol. 59, no. 5, pp.3046-3058, 2013.


\bibitem[19]{nagao} H. Nagao, Y. Tsushima. \textit{Representations of finite groups}. Academic Press Inc., 1989.


\bibitem[20]{grouprings} C. Polcino, S. K. Sehgal: \textit{An Introduction to Group Rings}. Kluwer academic publishers, 2002.







\end{thebibliography}
\end{document}